\documentclass[preprint,12pt]{elsarticle}
\usepackage{srcltx}
\usepackage{xcolor}
\usepackage{graphicx}
\usepackage{amsmath}
\usepackage{amssymb}
\usepackage{theorem}
\usepackage{euscript}
\usepackage{epic,eepic}
\usepackage{pstricks}
\definecolor{labelkey}{rgb}{0,0.08,0.45}
\definecolor{refkey}{rgb}{0,0.6,0.0}
\definecolor{Brown}{rgb}{0.45,0.0,0.05}
\definecolor{dgreen}{rgb}{0.00,0.49,0.00}
\definecolor{dblue}{rgb}{0,0.08,0.75}
\RequirePackage[dvipdfm,colorlinks,hyperindex]{hyperref}
\hypersetup{
    linktocpage=true,
    citecolor=dblue,    % color can be changed!
    linkcolor=dgreen    % color can be changed!
}
\newcommand{\scal}[2]{{\left\langle{{#1}\mid{#2}}\right\rangle}}

\newcommand{\menge}[2]{\big\{{#1}~\big |~{#2}\big\}}

\newcommand{\HH}{\ensuremath{{\mathcal H}}}
\newcommand{\AAA}{\ensuremath{{\mathcal A}}}

\newcommand{\spa}{\ensuremath{\operatorname{span}}}
\newcommand{\zer}{\ensuremath{\operatorname{zer}}}
\newcommand{\gra}{\ensuremath{\operatorname{gra}}}
\newcommand{\spc}{\ensuremath{\overline{\operatorname{span}}\,}}

\newcommand{\emp}{\ensuremath{{\varnothing}}}

\newcommand{\Id}{\ensuremath{\operatorname{Id}}\,}
\newcommand{\Fix}{\ensuremath{\operatorname{Fix}}\,}
\newcommand{\RR}{\ensuremath{\mathbb{R}}}

\newcommand{\RPP}{\ensuremath{\left]0,+\infty\right[}}

\newcommand{\RX}{\ensuremath{\left]-\infty,+\infty\right]}}

\newcommand{\NN}{\ensuremath{\mathbb N}}

\newcommand{\gr}{\ensuremath{\operatorname{gra}}}

\newcommand{\exi}{\ensuremath{\exists\,}}
\newcommand{\sri}{\ensuremath{\operatorname{sri}}}
\newcommand{\pinf}{\ensuremath{{+\infty}}}

\newcommand{\weakly}{\ensuremath{\:\rightharpoonup\:}}

\newcommand{\dom}{\ensuremath{\operatorname{dom}}}
\newcommand{\prox}{\ensuremath{\operatorname{prox}}}

\newcommand{\inte}{\ensuremath{\operatorname{int}}}

\newtheorem{theorem}{Theorem}[section]
\newtheorem{lemma}[theorem]{Lemma}

\newtheorem{proposition}[theorem]{Proposition}
\theoremstyle{plain}{\theorembodyfont{\rmfamily}%
\newtheorem{assumption}[theorem]{Assumption}}
\theoremstyle{plain}{\theorembodyfont{\rmfamily}%
}
\theoremstyle{plain}{\theorembodyfont{\rmfamily}%
\newtheorem{example}[theorem]{Example}}
\theoremstyle{plain}{\theorembodyfont{\rmfamily}%
\newtheorem{remark}[theorem]{Remark}}
\theoremstyle{plain}{\theorembodyfont{\rmfamily}%
}
\theoremstyle{plain}{\theorembodyfont{\rmfamily}%
\newtheorem{problem}[theorem]{Problem}}

\numberwithin{equation}{section}
\journal{Nonlinear Analysis: Theory, Methods \& Applications}

\begin{document}

\begin{frontmatter}

\title{\sffamily A DOUGLAS--RACHFORD SPLITTING METHOD FOR SOLVING 
EQUILIBRIUM PROBLEMS\tnoteref{title}}
\tnotetext[title]{This work was supported by CONICYT under grant 
FONDECYT N$^\text{o}$ 3120054.}
\author{Luis M. Brice\~{n}o-Arias\fnref{fono}}
\ead{lbriceno@dim.uchile.cl}
\fntext[fono]{Centro de Modelamiento Matem\'atico--CNRS-UMI 2807, 
Blanco Encalada 2120, Santiago, Chile, 7th floor, of. 706,
phone: +56 2 978 4930.}
\ead[url]{http://www.dim.uchile.cl/$\sim$ lbriceno}
\address{
\small Center for Mathematical Modeling--CNRS-UMI 2807\\ 
\small  University of Chile. 
}

\begin{abstract}
\noindent
We propose a splitting method for solving equilibrium problems
involving the sum of two bifunctions
satisfying standard conditions.
We prove that this problem is equivalent to find 
a zero of the sum of two appropriate maximally monotone operators
under a suitable qualification condition. 
Our algorithm is a consequence of the Douglas--Rachford 
splitting applied to this auxiliary monotone inclusion.
Connections between monotone inclusions and equilibrium 
problems are studied.
\end{abstract}

\begin{keyword}
Douglas--Rachford splitting\sep equilibrium problem\sep
maximally monotone operator\sep monotone inclusion

\MSC[2010] Primary: 90B74\sep 90C30. Secondary: 47J20\sep 47J22
\end{keyword}

\end{frontmatter}

\section{Introduction}
In the past years, several works have been devoted to
the equilibrium problem
\begin{equation}
\label{e:tipic}
\text{find}\quad x\in C\quad\text{such that}\quad (\forall y\in C)
\quad H(x,y)\geq0,
\end{equation}
where $C$ is a nonempty closed convex subset of the real Hilbert 
space $\HH$, and $H\colon C\times C\to\RR$ satisfies the following assumption.
\begin{assumption}
\label{a:1}
The bifunction $H\colon C\times C\to\RR$ satisfies
\begin{enumerate} 
\item \label{a:1i}$(\forall x\in C)\quad H(x,x)=0$.
\item \label{a:1ii}$(\forall (x,y)\in C\times C)\quad H(x,y)+
H(y,x)\leq0$.
\item \label{a:1iii} For every $x$ in $C$, $H(x,\cdot)\colon C\to\RR$ 
is lower semicontinuous and convex.
\item \label{a:1iv} $(\forall (x,y,z)\in C^3)\quad 
\underset{\varepsilon\to 0^+}{\varlimsup}\:
H((1-\varepsilon)x+\varepsilon z,y)\leq H(x,y)$.
\end{enumerate}
\end{assumption}
Throughout this paper, the solution set of \eqref{e:tipic} will be 
denoted by $S_{H}$.

Problem \eqref{e:tipic} models a wide variety of 
problems including complementarity problems, optimization problems,
feasibility problems, Nash equilibrium problems, variational
inequalities, and fixed point problems \cite{Alle77,Bian96,Blum94,%
ComH05,Flam97,Iuse03,Iuse03b,Konn03,Moud02,Oett97}. Sometimes the 
bifunction $H$ is difficult to manipulate
but it can be considered as the sum of two simpler bifunctions
$F$ and $G$ satisfying Assumption~\ref{a:1} (see, for example, 
\cite{Moud10}). This is the context in which we aim to solve 
problem \eqref{e:tipic}. Our problem is formulated as follows.
\begin{problem}
\label{prob:1}
Let $C$ be a nonempty closed convex subset of the real Hilbert 
space $\HH$. Suppose that $F\colon C\times C\to\RR$ and $G\colon C\times C\to\RR$
are two bifunctions satisfying Assumption~\ref{a:1}. The problem is to
\begin{equation}
\label{e:pmain}
\text{find}\quad x\in C\quad\text{such that}\quad
(\forall y\in C)\quad F(x,y)+G(x,y)\geq0,
\end{equation}
under the assumption that such a solution exists, or equivalently, 
$S_{F+G}\neq\emp$.
\end{problem}

In the particular instance when $G\equiv0$, Problem~\ref{prob:1} 
becomes \eqref{e:tipic} with $H=F$, which can be solved by the
methods proposed in \cite{Flam97,Iuse03,Konn03,Moud03,MouT99}.
These methods are mostly inspired from the proximal fixed point
algorithm \cite{Mart70,Rock76}. The method proposed in \cite{Reic11}
can be applied to this case when $F\colon(x,y)\mapsto\scal{Bx}{y-x}$
and $B$ is maximally monotone.
On the other hand, when $G\colon(x,y)\mapsto\scal{Bx}{y-x}$, where
$B\colon\HH\to\HH$ is a cocoercive operator, weakly convergent 
splitting methods for solving Problem~\ref{prob:1} are proposed in 
\cite{ComH05,Moud02}. Several methods for solving
Problem~\ref{prob:1} in particular instances of the bifunction $G$
can be found in \cite{Nfao1,Ceng08,Ceng10,Konn05,Peng10,Peng09,Peng09b,Yao09} 
and the references therein.
In the general case, sequential and parallel splitting methods are 
proposed in \cite{Moud09} with guaranteed ergodic convergence. 
A disadvantage of these methods is the involvement of vanishing
parameters that leads to numerical instabilities, 
which make them of limited use in applications. The purpose of this
paper
is to address the general case by providing a non-ergodic weakly 
convergent algorithm which solves Problem~\ref{prob:1}. The proposed 
method is a consequence of the Douglas-Rachford splitting method
\cite{Lion79,Svai11} 
applied to an auxiliary monotone inclusion involving an appropriate 
choice of maximally monotone operators.
This choice of monotone operators allows us to deduce interesting 
relations between monotone equilibrium problems and monotone inclusions 
in Hilbert spaces. Some of these relations are deduced from related
results in Banach spaces \cite{Aoya08,Saba11}.

The paper is organized as follows. In Section~\ref{sec:2}, we define
an auxiliary monotone inclusion which is equivalent to
Problem~\ref{prob:1} under a suitable qualification condition, 
and some relations between monotone inclusions and equilibrium
problems are examined. In Section~\ref{sec:3}, we propose a variant of
the Douglas--Rachford splitting studied in \cite{Livre1,Svai11} 
and we derive our method whose iterates converge weakly to a solution
of Problem~\ref{prob:1}. We start with some notation and useful
properties.

{\bf Notation and preliminaries}\\
Throughout this paper, $\HH$ denotes a real Hilbert space,
$\scal{\cdot}{\cdot}$ denotes its inner product, and $\|\cdot\|$ denotes
its induced norm. 
Let $\mathcal{A}\colon\HH\to 2^{\HH}$ be a set-valued operator. 
Then $\dom\mathcal{A}=\menge{x\in\HH}{\mathcal{A}x\neq\emp}$ is the
domain of $\mathcal{A}$ 
and $\gr\mathcal{A}=\menge{(x,u)\in\HH\times\HH}{u\in\mathcal{A}x}$ 
is its graph. The operator $\mathcal{A}$ is monotone if 
\begin{equation}
(\forall (x,u)\in\gr\mathcal{A})(\forall (y,v)\in\gr\mathcal{A})\quad
\scal{x-y}{u-v}\geq0,
\end{equation}
and it is called maximally monotone if its graph is not properly
contained in the graph of any other monotone operator in $\HH$.
In this case, the resolvent of $\mathcal{A}$,
$\mathcal{J}_{\mathcal{A}}=(\Id+\mathcal{A})^{-1}$, is well 
defined, single valued, and $\dom \mathcal{J}_{\mathcal{A}}=\HH$. The
reflection operator
$\mathcal{R}_{\mathcal{A}}=2\mathcal{J}_{\mathcal{A}}-\Id$ is
nonexpansive.

For a single-valued operator $T\colon\dom T\subset\HH\to\HH$, 
the set of fixed points is
\begin{equation}
\Fix T=\menge{x\in\HH}{x=Tx}.
\end{equation}
We say that $T$ is nonexpansive if
\begin{equation}
(\forall x\in\dom T)(\forall y\in\dom T)\quad
\|Tx-Ty\|\leq\|x-y\|
\end{equation}
and that $T$ is firmly nonexpansive if
\begin{equation}
\label{e:firmnonexp}
(\forall x\in\dom T)(\forall y\in\dom T)\quad
\|Tx-Ty\|^2\leq\|x-y\|^2-\|(\Id-T)x-(\Id-T)y\|^2.
\end{equation}
\begin{lemma}[{\em cf.} \mbox{\cite[Lemma~5.1]{Opti04}}]
\label{l:1}
Let $T\colon\dom T=\HH\to\HH$ be a nonexpansive operator such that
$\Fix T\neq\emp$. Let 
$(\mu_n)_{n\in\NN}$ be a sequence in $\left]0,1\right[$ and
$(c_n)_{n\in\NN}$ be a sequence in $\HH$ such that 
$\sum_{n\in\NN}\mu_n(1-\mu_n)=\pinf$ and 
$\sum_{n\in\NN}\mu_n\|c_n\|<\pinf$. Let $x_0\in\HH$ and
set
\begin{equation}
(\forall n\in\NN)\quad
x_{n+1}=x_n+\mu_n(Tx_n+c_n-x_n).
\end{equation}
Then $(x_n)_{n\in\NN}$ converges weakly to $x\in\Fix T$ and
$(x_n-Tx_n)_{n\in\NN}$ converges strongly to $0$.
\end{lemma}
Now let $F\colon C\times C\to\RR$ be a bifunction satisfying Assumption~\ref{a:1}. 
The resolvent of $F$ is the operator
\begin{equation}
\label{e:J_F}
J_F\colon \HH\to 2^C\colon x\mapsto\menge{z\in C}
{(\forall y\in C)\quad F(z,y)+\scal{z-x}{y-z}\geq 0},
\end{equation}
which is single valued and firmly nonexpansive \cite[Lemma~2.12]{ComH05}, and 
the reflection operator
\begin{equation}
\label{e:reflec}
R_F\colon\HH\to\HH\colon x\mapsto 2J_Fx-x
\end{equation}
is nonexpansive.

Let $C\subset\HH$ be nonempty, closed, and convex.
We say that $0$ lies in the strong relative interior of $C$, 
in symbol, $0\in\sri C$, if $\bigcup_{\lambda>0}\lambda C=\spc C$.
The normal cone of $C$ is the maximally monotone operator 
\begin{equation}
\mathcal{N}_C\colon\HH\to 2^{\HH}\colon x\mapsto
\begin{cases}
\menge{u\in\HH}{(\forall y\in
C)\quad\scal{y-x}{u}\leq0},\quad&\text{if } x\in C;\\
\emp,&\text{otherwise}. 
\end{cases}
\end{equation}

We denote by $\Gamma_0(\HH)$ the family of 
lower semicontinuous convex functions $f$ from $\HH$ to $\RX$
which are proper
in the sense that $\dom f=\menge{x\in\HH}{f(x)<\pinf}$ is nonempty.
The 
subdifferential of $f\in\Gamma_0(\HH)$ is the maximally monotone operator
$\partial f\colon\HH\to 2^{\HH}\colon x\mapsto
\menge{u\in\HH}{(\forall y\in\HH)\;\:\scal{y-x}{u}+f(x)\leq f(y)}$.
For background on convex analysis, monotone operator theory, and 
equilibrium problems, the reader is referred to
\cite{Livre1,Blum94,ComH05}.

\section{Monotone inclusions and equilibrium problems}
\label{sec:2}
The basis of the method proposed in this paper for solving 
Problem~\ref{prob:1} is that it can be formulated as finding
a zero of the sum of two appropriate maximally monotone operators.
In this section, we define this auxiliary monotone inclusion and,
additionally, we study a class of monotone inclusions which can be 
formulated as an equilibrium problem.

\subsection{Monotone inclusion associated to equilibrium problems}
We first recall the maximal monotone operator associated to problem 
\eqref{e:tipic} and some related properties. The 
following result can be deduced from \cite[Theorem~3.5]{Aoya08} and
\cite[Proposition~4.2]{Saba11}, which have been proved in Banach
spaces.
\begin{proposition}
\label{prop:A}
Let $F\colon C\times C\to\RR$ be such that Assumption~\ref{a:1} holds and set
\begin{equation}
\label{e:maxmon}
\AAA_{F}\colon \HH\to 2^{\HH}\colon x\mapsto
\begin{cases}
\menge{u\in\HH}{(\forall y\in C)\quad F(x,y)+\scal{x-y}{u}\geq0},\quad&\text{if }x\in C;\\
\emp,&\text{otherwise}.
\end{cases}
\end{equation}
Then the following hold:
\begin{enumerate}
\item \label{prop:Ai} $\AAA_F$ is maximally monotone.
\item $S_F=\zer\AAA_F$.
\item \label{prop:Aii} For every $\gamma\in\RPP$, 
$\mathcal{J}_{\gamma \AAA_F}=J_{\gamma F}$.
\end{enumerate}
\end{proposition}

The following proposition allows us to formulate Problem~\ref{prob:1}
as an auxiliary monotone inclusion involving two maximally
monotone operators obtained from
Proposition~\ref{prop:A}.

\begin{theorem}
\label{prop:1}
Let $C$, $F$, and $G$ be as in Problem~\ref{prob:1}. 
Then the following hold.
\begin{enumerate}
\item\label{prop:1i} $\zer(\AAA_F+\AAA_G)\subset S_{F+G}$.
\item\label{prop:1ii} Suppose that $\spa(C-C)$ is closed. Then, 
$\zer(\AAA_F+\AAA_G)= S_{F+G}$.
\end{enumerate}
\end{theorem}
\begin{proof}
\ref{prop:1i}.
Let $x\in\zer(\AAA_F+\AAA_G)$. Thus, $x\in C$ and 
there exists $u\in \AAA_Fx\cap-\AAA_Gx$, which yield, 
by \eqref{e:maxmon},
\begin{equation}
\begin{cases}
(\forall y\in C)\quad F(x,y)+\scal{x-y}{u}\geq0\\
(\forall y\in C)\quad G(x,y)+\scal{x-y}{-u}\geq0.
\end{cases}
\end{equation}
Hence, by adding both inequalities we obtain 
\begin{equation}
(\forall y\in C)\quad  F(x,y)+G(x,y)\geq0
\end{equation}
and, therefore, $x\in S_{F+G}$.

\ref{prop:1ii}. Let $x\in S_{F+G}$ and define
\begin{equation}
\label{e:f1f2}
\begin{cases}
f\colon\HH\to\RX\colon y\mapsto
\begin{cases}
F(x,y),\quad&\text{if}\:\: y\in C;\\
\pinf,&\text{otherwise};
\end{cases}\vspace{0.3cm}\\
g\colon\HH\to\RX\colon y\mapsto
\begin{cases}
G(x,y),\quad&\text{if}\:\: y\in C;\\
\pinf,&\text{otherwise}.
\end{cases}\\
\end{cases}
\end{equation}
Assumption~\ref{a:1} asserts that $f$ and $g$
are in $\Gamma_0(\HH)$, $\dom f=\dom g=C\neq\emp$, and since 
$x\in S_{F+G}$, \eqref{e:pmain} yields $f+g\geq0$.
Hence, it follows from Assumption~\ref{a:1}\ref{a:1i} and 
\eqref{e:f1f2} that
\begin{equation}
\label{e:eq333}
\min_{y\in\HH}\,\big(f(y)+g(y)\big)=f(x)+g(x)=0.
\end{equation}
Thus, Fermat's rule \cite[Theorem~16.2]{Livre1} yields
$0\in\partial(f+g)(x)$. Since $\spa(C-C)$ is closed, we have 
$0\in\sri(C-C)=\sri(\dom f-\dom g)$. Therefore, it follows from 
\cite[Corollary~16.38]{Livre1} that 
$0\in\partial f(x)+\partial g(x)$ which implies that there exists 
$u_0\in\HH$ such that $u_0\in\partial f(x)$ and 
$-u_0\in\partial g(x)$. This is equivalent to 
\begin{equation}
\label{e:desig}
\begin{cases}
(\forall y\in\HH)\quad f(x)+\scal{y-x}{u_0}\leq f(y)\\
(\forall y\in\HH)\quad g(x)+\scal{y-x}{-u_0}\leq g(y).
\end{cases}
\end{equation}
Since Assumption~\ref{a:1}\ref{a:1i} and \eqref{e:f1f2} 
yield $f(x)=g(x)=0$, we have that \eqref{e:desig} is equivalent to
\begin{equation}
\begin{cases}
(\forall y\in C)\quad F(x,y)+\scal{x-y}{u_0}\geq0 \\
(\forall y\in C)\quad G(x,y)+\scal{x-y}{-u_0}\geq0.
\end{cases}
\end{equation}
Hence, we conclude from \eqref{e:maxmon} that 
$u_0\in \AAA_Fx\cap-\AAA_Gx$, which yields 
$x\in\zer(\AAA_F+\AAA_G)$.
\end{proof}

\subsection{Equilibrium problems associated to monotone inclusions}
We formulate some monotone inclusions as equilibrium problems 
by defining a bifunction associated to a class of maximally monotone operators.
In the following proposition we present this bifunction and its properties.
\begin{proposition}{\rm({\em cf.} \cite[Lemma~2.15]{ComH05})}
\label{prop:22}
Let $\mathcal{A}\colon\HH\to 2^{\HH}$ be a maximally monotone operator
and
suppose that $C\subset\inte\dom\mathcal{A}$. 
Set 
\begin{equation}
\label{e:FFF}
F_{\mathcal{A}}\colon C\times C\to\RR\colon (x,y)\mapsto\max_{u\in
\mathcal{A}x}\scal{y-x}{u}.  
\end{equation}
Then the following hold:
\begin{enumerate}
\item $F_{\mathcal{A}}$ satisfy Assumption~\ref{a:1}.
\item $J_{F_{\mathcal{A}}}=\mathcal{J}_{\mathcal{A}+\mathcal{N}_C}$.
\end{enumerate}
\end{proposition}

\begin{remark}
Note that the condition $C\subset\inte\dom\mathcal{A}$ allows us to
take 
the maximum in \eqref{e:FFF} instead of the supremum. This is 
a consequence of the weakly compactness of the sets
$(\mathcal{A}x)_{x\in C}$
(see \cite[Lemma~2.15]{ComH05} for details).
\end{remark}

\begin{proposition}
\label{prop:equiv}
Let $\mathcal{A}\colon\HH\to 2^{\HH}$ be a maximally monotone operator
and
suppose that $C\subset\inte\dom\mathcal{A}$. Then
$\zer(\mathcal{A}+\mathcal{N}_C)=S_{F_{\mathcal{A}}}$.
\end{proposition}
\begin{proof}
Indeed, it follows from \cite[Proposition~23.38]{Livre1}, 
Proposition~\ref{prop:22}, and \cite[Lemma~2.15(i)]{ComH05} that
\begin{equation}
\zer(\mathcal{A}+\mathcal{N}_C)=\Fix(\mathcal{J}_{\mathcal{A}+
\mathcal{N}_C} )=\Fix(J_{ F_{\mathcal{A}}})=S_{F_{\mathcal{A}}}, 
\end{equation}
which yields the result.
\end{proof}

\begin{remark}\
\begin{enumerate}
\item Note that, in the particular case when
$\dom\mathcal{A}=\inte\dom\mathcal{A}=C=\HH$,
Proposition~\ref{prop:equiv} asserts that
$\zer\mathcal{A}=S_{F_{\mathcal{A}}}$, 
which is a well known result (e.g., see \cite[Section~2.1.3]{Konn01}).

\item In Banach spaces, the case when $C=\dom\mathcal{A}\subset\HH$ is
studied in \cite[Theorem~3.8]{Aoya08}.
\end{enumerate}
\end{remark}

The following propositions provide a relation between 
the operators defined in Propositions~\ref{prop:A} and \ref{prop:22}.

\begin{proposition}
\label{prop:3}
Let $\mathcal{B}\colon\HH\to 2^{\HH}$ be maximally monotone and
suppose that $C\subset\inte\dom\mathcal{B}$. Then, 
$\AAA_{F_{\mathcal{B}}}=\mathcal{B}+\mathcal{N}_C$.
\end{proposition}
\begin{proof}
Let $(x,u)\in\HH^2$. It follows from \eqref{e:FFF} and 
\cite[Lemma~1]{Blum94} (see also \cite[Lemma~2.14]{ComH05}) that
\begin{align}
u\in\AAA_{F_{\mathcal{B}}}x\quad&\Leftrightarrow\quad x\in C\quad
\text{and}\quad (\forall y\in C)\quad
F_{\mathcal{B}}(x,y)+\scal{x-y}{u}\geq0\nonumber\\
&\Leftrightarrow\quad x\in C\quad
\text{and}\quad (\forall y\in C)\quad
\max_{v\in\mathcal{B}x}\scal{y-x}{v}+\scal{x-y}{u}\geq0\\
&\Leftrightarrow\quad x\in C\quad
\text{and}\quad (\forall y\in C)\quad
\max_{v\in\mathcal{B}x}\scal{y-x}{v-u}\geq0\\
&\Leftrightarrow\quad x\in C\quad
\text{and}\quad (\exi v\in\mathcal{B}x)(\forall y\in C)\quad
\scal{y-x}{v-u}\geq0\\
&\Leftrightarrow\quad u\in\mathcal{B}x+\mathcal{N}_Cx,
\end{align}
which yields the result.
\end{proof}

\begin{proposition}
\label{prop:4}
Let $G$ be such that Assumption~\ref{a:1} holds, and suppose that
$C=\dom\mathcal{A}_G=\HH$. Then, $F_{\AAA_G}\leq G$.
\end{proposition}
\begin{proof}
Let $(x,y)\in C\times C$ and let $u\in\AAA_Gx$. It
follows from
\eqref{e:maxmon} that $G(x,y)+\scal{x-y}{u}\geq0$,
which yields
\begin{equation}
(\forall u\in \AAA_Gx)\quad \scal{y-x}{u}\leq G(x,y).
\end{equation}
Since $C=\inte\dom\AAA_G=\HH$, the result follows by taking the
maximum in the left side of the inequality.
\end{proof}

\begin{remark}\
\begin{enumerate}
 \item 
Note that the equality in Proposition~\ref{prop:4}
does not hold in general. Indeed, let $\HH=\RR$, $C=\HH$, and
$G\colon(x,y)\mapsto y^2-x^2$. It follows from
\cite[Lemma~2.15(v)]{ComH05} that $G$ satisfy Assumption~\ref{a:1}.
We have $u\in\AAA_Gx$ $\Leftrightarrow$
$(\forall y\in\HH)\quad y^2-x^2+\scal{x-y}{u}\geq0$ $\Leftrightarrow$
$u=2x$ and, hence, for every $(x,y)\in\HH\times\HH$, $F_{\AAA_G}(x,y)
=(y-x)2x=2xy-2x^2$. In particular, for every $y\in\RR\setminus\{0\}$, 
$F_{\AAA_G}(0,y)=0<y^2=G(0,y)$.  
\item In the general case when $C=\dom\mathcal{A}\subset\HH$,
necessary and sufficient conditions for the equality in
Proposition~\ref{prop:4} are provided in \cite[Theorem~4.5]{Aoya08}.
\end{enumerate}

\end{remark}

\section{Algorithm and convergence}
\label{sec:3}

Theorem~\ref{prop:1}\ref{prop:1ii} characterizes the solutions to 
Problem~\ref{prob:1} as the zeros of the sum of two maximally 
monotone operators. Our algorithm is derived from the 
Douglas-Rachford splitting method for solving this auxiliary monotone 
inclusion. This algorithm was first proposed in \cite{Doug56} 
in finite dimensional spaces when the operators are linear and the 
generalization to general maximally monotone operators in Hilbert 
spaces was first developed in \cite{Lion79}.
Other versions involving computational errors of the resolvents
can be found in \cite{Opti04,Ecks92}. The convergence of these methods 
needs the maximal monotonicity of 
the sum of the operators involved, which is not evident to verify
\cite[Section~24.1]{Livre1}. Furthermore, the iterates in these cases 
do not converge to a solution but to a point from which we can calculate 
a solution. These problems were overcame in \cite{Svai11} and, 
later, in \cite[Theorem~25.6]{Livre1}, where the convergence of the sequences
generated by the proposed methods to a zero of the sum of two set-valued 
operators is guaranteed by only assuming the maximal monotonicity of 
each operator. However, in \cite{Svai11} the errors considered do not come 
from inaccuracies on the computation of the resolvent but only from 
imprecisions in a monotone inclusion, which sometimes could be not manipulable.
On the other hand, in \cite[Theorem~25.6]{Livre1} the method includes  
an additional relaxation step but it does not consider inaccuracies in 
its implementation.

We present a variant of the methods presented in 
\cite{Svai11} and \cite[Theorem~25.6]{Livre1}, which has interest
in its own right. The same convergence results are obtained by 
considering a relaxation step as in \cite{Ecks92,Ecks09} and 
errors in the computation of the resolvents as in \cite{Opti04,Ecks92}. 

\begin{theorem}
\label{p:SvaiterEP}
Let $\mathcal{A}$ and $\mathcal{B}$ be two maximally monotone
operators from $\HH$ to $2^{\HH}$ such that
$\zer(\mathcal{A}+\mathcal{B})\neq\emp$. Let $\gamma\in\RPP$, let 
$(\lambda_n)_{n\in\NN}$ be a sequence in $\left]0,2\right[$, and
let $(a_n)_{n\in\NN}$ and $(b_n)_{n\in\NN}$ be sequences in 
$\HH$ such that $b_n\weakly0$,
\begin{equation}
\label{e:aux1svai}
\sum_{n\in\NN}\lambda_n(2-\lambda_n)=\pinf,\quad\text{and}
\quad \sum_{n\in\NN}\lambda_n(\|a_n\|+\|b_n\|)<\pinf.
\end{equation}
Let $x_0\in\HH$ and set
\begin{align}
\label{e:DR0}
(\forall n\in\NN)\quad
% &\left \lfloor \begin{array}{l}
\begin{cases}
y_n=\mathcal{J}_{\gamma\mathcal{B}}x_n+b_n\\
z_n=\mathcal{J}_{\gamma\mathcal{A}}(2y_n-x_n)+a_n\\
x_{n+1}=x_n+\lambda_n(z_n-y_n).
\end{cases}
% \end{array}
% \right.
\end{align}
Then there exists
$x\in\Fix(\mathcal{R}_{\gamma\mathcal{A}}
\mathcal{R}_{\gamma\mathcal{B}})$ such that 
the following hold:
\begin{enumerate}
\item\label{p:SvaiterEPi} 
$\mathcal{J}_{\gamma\mathcal{B}}x\in\zer(\mathcal{A}+\mathcal{B})$.
\item\label{p:SvaiterEPii} 
$(\mathcal{R}_{\gamma\mathcal{A}}
(\mathcal{R}_{\gamma\mathcal{B}}x_n)-x_n)_{n\in\NN}$ converges 
strongly to $0$.
\item\label{p:SvaiterEPiii} 
$(x_n)_{n\in\NN}$ converges weakly to $x$.
\item\label{p:SvaiterEPiv} 
$(y_n)_{n\in\NN}$ converges weakly to
$\mathcal{J}_{\gamma\mathcal{B}}x$.
\end{enumerate}
\end{theorem}
\begin{proof}
Denote
$\mathcal{T}=\mathcal{R}_{\gamma\mathcal{A}}
\mathcal{R}_{\gamma\mathcal {B}}$. Since
$\mathcal{R}_{\gamma\mathcal{A}}$ and 
$\mathcal{R}_{\gamma\mathcal{B}}$ are nonexpansive operators,
$\mathcal{T}$ is nonexpansive as 
well. Moreover, since \cite[Proposition~25.1(ii)]{Livre1} states 
that $\mathcal{J}_{\gamma\mathcal{B}}(\Fix\mathcal{T})=
\zer(\mathcal{A}+\mathcal{B})$, we
deduce that $\Fix\mathcal{T}\neq\emp$. Note that \eqref{e:DR0} can be
rewritten 
as
\begin{equation}
(\forall n\in\NN)\quad x_{n+1}=x_n+\mu_n(\mathcal{T}x_n+c_n-x_n),
\end{equation}
where, for every $n\in\NN$,
\begin{equation}
\label{e:muerr}
\mu_n=\frac{\lambda_n}{2}\quad\text{and}\quad
c_n=2\big(\mathcal{J}_{\gamma\mathcal{A}}
\big(2(\mathcal{J}_{\gamma\mathcal{B}}x_n+b_n)-x_n\big)+a_n- 
\mathcal{J}_{\gamma\mathcal{A}}
(2\mathcal{J}_{\gamma\mathcal{B}}x_n-x_n)-b_n\big).
\end{equation}
Hence, it follows from the nonexpansivity of 
$\mathcal{J}_{\gamma\mathcal{A}}$ that, for every $n\in\NN$,
\begin{align}
\|c_n\|&=2\|\mathcal{J}_{\gamma\mathcal{A}}
\big(2(\mathcal{J}_{\gamma\mathcal{B}}x_n+b_n)-x_n\big)+a_n- 
\mathcal{J}_{\gamma\mathcal{A}}
(2\mathcal{J}_{\gamma\mathcal{B}}x_n-x_n)-b_n\|\nonumber\\
&\leq 2\|\mathcal{J}_{\gamma\mathcal{A}}
\big(2(\mathcal{J}_{\gamma\mathcal{B}}x_n+b_n)-x_n\big)- 
\mathcal{J}_{\gamma\mathcal{A}}
(2\mathcal{J}_{\gamma\mathcal{B}}
x_n-x_n)\|+2\|a_n\|+2\|b_n\|\nonumber\\
&\leq 2\|2(\mathcal{J}_{\gamma\mathcal{B}}x_n+b_n)-x_n-
(2\mathcal{J}_{\gamma\mathcal{B}}x_n-x_n)\|
+2\|a_n\|+2\|b_n\|\nonumber\\
&= 2(\|a_n\|+3\|b_n\|)
\end{align}
and, therefore, from \eqref{e:aux1svai} and \eqref{e:muerr} we obtain
\begin{equation}
\sum_{n\in\NN}\mu_n\|c_n\|\leq
\sum_{n\in\NN}\lambda_n(\|a_n\|+3\|b_n\|)
\leq 3\sum_{n\in\NN}\lambda_n(\|a_n\|+\|b_n\|)<\pinf. 
\end{equation}
Moreover, since the sequence $(\lambda_n)_{n\in\NN}$ is in
$\left]0,2\right[$, it follows from \eqref{e:muerr} that
$(\mu_n)_{n\in\NN}$ is a sequence in $\left]0,1\right[$ and, 
from \eqref{e:aux1svai} we obtain
\begin{equation}
\sum_{n\in\NN}\mu_n(1-\mu_n)=\frac14\sum_{n\in\NN}
\lambda_n(2-\lambda_n)=\pinf.
\end{equation}

\ref{p:Svaiteri}. 
This follows from \cite[Proposition~25.1(ii)]{Livre1}.

\ref{p:Svaiterii} and \ref{p:Svaiteriii}. These follow from 
Lemma~\ref{l:1}.
 
\ref{p:Svaiteriv}. From the nonexpansivity 
of $\mathcal{J}_{\gamma\mathcal{B}}$ we obtain
\begin{equation}
\label{e:yboun}
\|y_n-y_0\|
\leq\|\mathcal{J}_{\gamma\mathcal{B}}x_n-
\mathcal{J}_{\gamma\mathcal{B}}x_0\|+\|b_n-b_0\|
\leq\|x_n-x_0\|+\|b_n-b_0\|.
\end{equation}
It follows from \ref{p:Svaiteriii} and $b_n\weakly0$ that 
$(x_n)_{n\in\NN}$ and $(b_n)_{n\in\NN}$ are bounded, respectively. 
Hence, \eqref{e:yboun} implies that $(y_n)_{n\in\NN}$ is bounded 
as well. Let $y\in\HH$ be a weak sequential cluster
point of $(y_n)_{n\in\NN}$, say $y_{k_n}\weakly y$, and set 
\begin{equation}
\label{e:tildas}
(\forall n\in\NN)\quad
\begin{cases}
\widetilde{y}_n=\mathcal{J}_{\gamma\mathcal{B}}x_n\\
\widetilde{z}_n=\mathcal{J}_{\gamma\mathcal{A}}(2\widetilde{y}
_n-x_n)\\
\widetilde{u}_n=2\widetilde{y}_n-x_n-\widetilde{z}_n\\
\widetilde{v}_n=x_n-\widetilde{y}_n.
\end{cases}
\end{equation}
It follows from \eqref{e:DR0} that
\begin{equation}
\label{e:tildas2}
(\forall n\in\NN)\quad
\begin{cases}
(\widetilde{z}_n,\widetilde{u}_n)\in\gra\gamma\mathcal{A}\\
(\widetilde{y}_n,\widetilde{v}_n)\in\gra\gamma\mathcal{B}\\
\widetilde{u}_n+\widetilde{v}_n=\widetilde{y}_n-\widetilde{z}_n.
\end{cases} 
\end{equation}
For every $n\in\NN$, we obtain from \eqref{e:tildas}
\begin{align}
\|\widetilde{z}_{k_n}-\widetilde{y}_{k_n}\|
&=\|\mathcal{J}_{\gamma\mathcal{A}}
(2\mathcal{J}_{\gamma\mathcal{B}}x_{k_n}-x_{k_n})-
\mathcal{J}_{\gamma\mathcal{B}}x_{k_n}\|
\nonumber\\
&=\frac12\|2\mathcal{J}_{\gamma\mathcal{A}}
(2\mathcal{J}_{\gamma\mathcal{B}}x_{k_n}-x_{k_n})
-(2\mathcal{J}_{\gamma\mathcal{B}}x_{k_n}-x_{k_n})-x_{k_n}
\|\nonumber\\
&=\frac12\|\mathcal{R}_{\gamma\mathcal{A}}
(\mathcal{R}_{\gamma\mathcal{B}}x_{k_n})-x_{k_n}\|.
\end{align}
Hence, \ref{p:Svaiterii} yields 
$\widetilde{z}_{k_n}-\widetilde{y}_{k_n}\to0$, and, therefore,
from \eqref{e:tildas2} we obtain that
$\widetilde{u}_{k_n}+\widetilde{v}_{k_n}\to0$.
Moreover, it follows from $b_{k_n}\weakly0$, $y_{k_n}\weakly y$, and 
\eqref{e:DR0} that $\widetilde{y}_{k_n}\weakly y$, and, hence,
$\widetilde{z}_{k_n}\weakly y$. 
Thus, from \ref{p:Svaiteriii} and \eqref{e:tildas},
we obtain $\widetilde{u}_{k_n}\weakly y-x$ and 
$\widetilde{v}_{k_n}\weakly x-y$. Altogether, from 
\cite[Corollary~25.5]{Livre1} we deduce that
$y\in\zer(\gamma\mathcal{A}+\gamma\mathcal{B})=\zer(\mathcal{A}
+\mathcal{B})$, $(y,y-x)\in\gra\gamma\mathcal{A}$,
and $(y,x-y)\in\gra\gamma\mathcal{B}$. Hence, $y=\mathcal{J}_{\gamma
\mathcal{B}}x$ and $y\in\dom\mathcal{A}$. Therefore, we conclude that
$\mathcal{J}_{\gamma\mathcal{B}}x$ is the 
unique weak sequential cluster point of $(y_n)_{n\in\NN}$ and then
$y_n\weakly\mathcal{J}_{\gamma\mathcal{B}}x$.
\end{proof}

Now we present our method for solving Problem~\ref{prob:1},
which is an application of Theorem~\ref{p:SvaiterEP} to
the auxiliary monotone inclusion obtained in Theorem~\ref{prop:1}.

\begin{theorem}
\label{t:main}
Let $C$, $F$, and $G$ be as in Problem~\ref{prob:1} and suppose 
that $\spa(C-C)$ is closed. Let $\gamma\in\RPP$, let 
$(\lambda_n)_{n\in\NN}$ be a sequence in $\left]0,2\right[$, and
let $(a_n)_{n\in\NN}$ and $(b_n)_{n\in\NN}$ be sequences in 
$\HH$ such that $b_n\weakly0$,
\begin{equation}
\label{e:aux2}
\sum_{n\in\NN}\lambda_n(2-\lambda_n)=\pinf,\quad\text{and}
\quad \sum_{n\in\NN}\lambda_n(\|a_n\|+\|b_n\|)<\pinf.
\end{equation}
Let $x_0\in\HH$ and set
\begin{align}
\label{e:DR0ep}
(\forall n\in\NN)\quad
% &\left \lfloor \begin{array}{l}
\begin{cases}
y_n=J_{\gamma G}x_n+b_n\\
z_n=J_{\gamma F}(2y_n-x_n)+a_n\\
x_{n+1}=x_n+\lambda_n(z_n-y_n).
\end{cases}
% \end{array}
% \right.
\end{align}
Then there exists $x\in\Fix(R_{\gamma F}R_{\gamma G})$ such that 
the following hold:
\begin{enumerate}
\item\label{p:Svaiteri} 
$J_{\gamma G}x\in S_{F+G}$.
\item\label{p:Svaiterii} 
$(R_{\gamma F}(R_{\gamma G}x_n)-x_n)_{n\in\NN}$ converges 
strongly to $0$.
\item\label{p:Svaiteriii} 
$(x_n)_{n\in\NN}$ converges weakly to $x$.
\item\label{p:Svaiteriv} 
$(y_n)_{n\in\NN}$ converges weakly to $J_{\gamma G}x$.
\end{enumerate}
\end{theorem}
\begin{proof}
Note that, from Theorem~\ref{prop:1}\ref{prop:1ii}, we have that 
\begin{equation}
\label{e:condreg}
\zer(\AAA_F+\AAA_G)=S_{F+G}\neq\emp, 
\end{equation}
where $\AAA_F$ and $\AAA_G$ are defined in \eqref{e:maxmon} and
maximally monotone by Proposition~\ref{prop:A}\ref{prop:Ai}.
In addition, it follows from Proposition~\ref{prop:A}\ref{prop:Aii}
that \eqref{e:DR0ep} can be written equivalently as \eqref{e:DR0}
with $\mathcal{A}=\AAA_F$ and $\mathcal{B}=\AAA_G$. Hence, the results
are derived from Theorem~\ref{p:SvaiterEP},
Proposition~\ref{prop:A}, and Theorem~\ref{prop:1}.
\end{proof}

\begin{remark}
Note that the closeness of $\spa(C-C)$ and 
Theorem~\ref{prop:1}\ref{prop:1ii} yields \eqref{e:condreg}, which 
allows us to apply Theorem~\ref{p:SvaiterEP} for obtaining our result.
However, it is well known that this qualification condition does not
always hold in infinite dimensional spaces. In such cases, 
it follows from Theorem~\ref{prop:1}\ref{prop:1i} that 
Theorem~\ref{t:main} still holds if $\zer(\AAA_F+\AAA_G)\neq\emp$.
Conditions for assuring existence of solutions to monotone inclusions can be 
found in \cite[Proposition~3.2]{Nash} and \cite{Livre1}.
\end{remark}

Finally, let us show an application of Theorem~\ref{t:main} for solving
mixed equilibrium problems. Let $f\in\Gamma_0(\HH)$. 
For every $x\in\HH$, $\prox_fx$ is the unique minimizer of the strongly 
convex function $y\mapsto f(y)+\|y-x\|^2/2$. The operator 
$\prox_f\colon\HH\to\HH$ thus defined is called the proximity operator.

\begin{example}
In Problem~\ref{prob:1}, suppose that $G\colon(x,y)\mapsto f(y)-f(x)$,
where $f\in\Gamma_0(\HH)$ is such that $C\subset\dom f$. Then 
Problem~\ref{prob:1} becomes
\begin{equation}
\label{e:mep}
\text{find}\quad x\in C\quad\text{such that}\quad (\forall y\in C)\quad
F(x,y)+f(y)\geq f(x), 
\end{equation}
which is known as a mixed equilibrium problem. This problem arises
in several applied problems and it can be solved by using 
some methods developed in \cite{Ceng08,Peng09b,Peng10,Yao09}. 
However, all this methods consider implicit steps involving
simultaneously $F$ and $f$, which is not easy to compute in 
general. On the other hand, it follows from
\cite[Lemma~2.15(v)]{ComH05} that \eqref{e:DR0ep} becomes 
\begin{equation}
\label{e:DR0mep}
(\forall n\in\NN)\quad
% \left \lfloor \begin{array}{l}
\begin{cases}
y_n=\prox_{\iota_C+\gamma f}x_n+b_n\\
z_n=J_{\gamma F}(2y_n-x_n)+a_n\\
x_{n+1}=x_n+\lambda_n(z_n-y_n),
\end{cases}
% \end{array}
% \right.
\end{equation}
which computes
separately the resolvent of $F$ and the proximity operator of $f$.
If $\spa(C-C)$ is closed, Theorem~\ref{t:main} ensures the weak 
convergence of the iterates of this method to a solution to 
\eqref{e:mep}. Examples of computable proximity operators and
resolvents of bifunctions can be 
found in \cite{Smms05} and \cite{ComH05}, respectively.
\end{example}
\section{Acknowledgement}
I thank Professor Patrick L. Combettes for bringing this problem to my 
attention and for helpful discussions. In addition, I would like to
thank the anonymous reviewers for their comments that help improve the
manuscript.

\end{document}